\documentclass[letterpaper]{article}
\usepackage{aaai16}
\usepackage{times}
\usepackage{helvet}
\usepackage{courier}
\usepackage{amsmath,amssymb,amsthm}
\usepackage{algorithm}
\usepackage{bm}
\usepackage{algpseudocode}
\usepackage{graphicx} % more modern
\usepackage{subfigure}
\usepackage{multirow}
\usepackage{multicol}
\usepackage{amsfonts}
\usepackage{thmtools}
\usepackage[numbers]{natbib}
\def\norm[#1]{\left\lVert#1\right\rVert}

\newtheorem{thm}{Theorem}
\newtheorem{lem}[thm]{Lemma}
\newtheorem{definition}[thm]{Definition}
\newtheorem{cor}[thm]{Corollary}

\newcommand{\SetR}{\mathbb{R}}

\newcommand{\Expect}{\mathbb{E}}
\newcommand{\Variance}{\mathbb{V}}

\title{Accelerating Random Kaczmarz Algorithm Based on Clustering Information}
\author{Yujun Li \and Kaichun Mo \and Haishan Ye \\
Department of Computer Science and Engineering\\
Shanghai Jiao Tong University \\
\{liyujun145,yhs12354123\}@gmail.com, daerduomkch@sjtu.edu.cn
}
\date{}

\begin{document}
\maketitle
\begin{abstract}

Kaczmarz algorithm is an efficient iterative algorithm to solve overdetermined consistent system of linear equations. 
During each updating step, Kaczmarz chooses a hyperplane based on an individual equation    
and projects the current estimate for the exact solution onto that space to get a new estimate.
Many vairants of Kaczmarz algorithms are proposed on how to choose better hyperplanes.
Using the property of randomly sampled data in high-dimensional space,
we propose an accelerated algorithm based on clustering information to improve block Kaczmarz and Kaczmarz via Johnson-Lindenstrauss lemma.
Additionally, we theoretically demonstrate convergence improvement on block Kaczmarz algorithm.

\end{abstract}

%%%%%%%%%%%%%%%%%%%%%%%%%%%%%%%%%%%%%%%%%%%%%%%%%%%
\section{Introduction}
%%%%%%%%%%%%%%%%%%%%%%%%%%%%%%%%%%%%%%%%%%%%%%%%%%%

In many real applications, we will face with solving the overdetermined consistent system of linear equations of the form $Ax = b$, 
where $A\in \SetR^{n\times p}$ and $b\in\SetR^{n}$ are given data, and $x\in\SetR^p$ is unknown vector to be estimated.
If $A$ has a small size, we could directly solve the problem by computing pseudo-inverse, $x = A^\dagger b$.
However, if $A$ is of large size, either we cannot store it in the main memory or it is extremely time-consuming to compute the pseudo-inverse.
Fortunately, in such cases, the Kaczmarz algorithm can be used to solve the problem, since we can update our current estimate to the exact solution $x_*$ iteratively by only using a small fraction of entire data each time.

Recently, many variants of the classical Kaczmarz algorithm \cite{kaczmarz1937angenaherte} are proposed by researchers. 
Classical Kaczmarz algorithm performs each iterative step by selecting rows of A in a sequential order. 
Despite the power of it, theoretical guarantee for its rate of convergence is scarce. 
However, a randomized version of Kaczmarz algorithm \cite{strohmer2009randomized}, denoted as RKA in this paper,
yields provably exponential rate of convergence in expectation by sampling rows of $A$ at random, 
with probability proportional to Euclidean norm. Recently,
more accelerated variants of Kaczmarz algorithms are put forth by researchers. 
With the usage of well-known Johnson-Lindenstrauss lemma \cite{johnson1984extensions}, 
Eldar and Needell proposed RKA-JL algorithm \cite{eldar2011acceleration}, which selects the optimal 
update from a randomly chosen set of linear equations during each iteration. 
Empirical studies demonstrate an improved rate of convergence than former methods. 
Moreover, along another direction of researches on accelerating Kaczmarz algorithm, 
multiple rows of A are utilized at the same time to perform one single updating step. 
These RKA-Block methods \cite{elfving1980block,needell2014paved,briskman2014block} use subsets of 
constraints during iterations so that an expected linear rate of convergence can
be provably achieved if utilizing randomization as well. However, geometric properties of blocks are important. 
Blocks with bad-condition number will interfere the rate of convergence. 

Since Kaczmarz based algorithms are widely used nowadays \cite{natterer1986mathematics,sezan1987incorporation,dai2014randomized,thoppe2014stochastic}, 
many other branches of research on Kaczmarz are developed. 
In terms of inconsistent case, many useful methods \cite{needell2010randomized,zouzias2013randomized} are proposed.
When $A$ is a large sparse matrix, there will appear a big fill-in, modified Kaczmarz via clustering Jaccard and Hamming distances can overcome this problem \cite{pomparuau2014supplementary,pomparuau2013acceleration}.
However in this paper, we consider solving consistent system of linear equations with high-dimensional matrix $A$ following Gaussian distribution.

One greedy idea highlighted by \cite{eldar2011acceleration} emphasize the importance of choosing hyperplanes 
that give the furthest distance to the current estimate during updating steps. 
Many methods such as \cite{eldar2011acceleration} are proposed to approximately utilize this idea by considering acceptable number of hyperplanes and selecting the furthest one. 
In this paper, we utilize clustering methods to gain more insight on where is the furthest hyperplane. 
Eventually, we incorporated clustering algorithms into one version of randomized Kaczmarz algorithm(RKA-JL) 
and block Kaczmarz algorithm(RKA-Block) in order to better approximate that optimal plane.

It is well-known that high-dimensioanl data points randomly sampled according to Gaussian distribution 
are much likely be orthogonal to each other. One can refer to \cite{hopcroft2015foundation} and Theorem \ref{thm:orthogonal-in-high-dimension} for details. 
Even though real data entities in high dimension are not ideally sampled from Gaussian distribution, 
they still tend to stretch along with different axes. In this paper, we cluster rows of $A$ into 
different classes so that the center points of clusters tend to be orthogonal.
After capturing the clustering information, we may first measure the distances between 
the centroids of clusters to the current estimate within a tolerable amount of time to
gain knowledge about distances from the current estimate to all hyperplanes determined by rows of $A$. 
Then, we may further select one sample from the furthest cluster to approximate the unknown furthest hyperplane.

The paper offers the following contributions: 

\begin{itemize}
 \item After applying clustering method and utilizing clustering information, we improve RKA-JL and RKA-Block algorithms to speedup their convergences. The empirical experiments show the improvement clearly.
 \item Theoretically, we prove that clustering method could improve convergence of RKA-Block algorithm. In addition, we coarsely analyze the modified RKA-JL (RKA- Cluster-JL) in terms of runtime. 
\end{itemize}

The remainder of the paper are organized as follows. 
In section 2, we give a short overview of two relevant algorithms proposed in recent years, which should be helpful in understanding our algorithms. 
Section 3 shows how to use clustering method to improve RKA-JL algorithm and RKA-Block algorithm. 
We prove theoretically that clustering method could improve convergence of RKA-Block algorithm.
In section 4, we conduct some numerical experiments to show the improved performance of our algorithms. 
Section 5 concludes the paper briefly.

\section{Related Work}
%%%%%%%%%%%%%%%%%%%%%%%%%%%%%%%%%%%%%%%%%%%%%%%%%%

Consider the overdetermined consistent linear equation system $Ax = b$,
where $A \in \SetR^{n\times p}$, $b\in\SetR^{n}$ and $x\in\SetR^{p}$.
Denote $x_0$ the initial guess of $x$, $A_1, A_2, ..., A_n$ the rows of data matrix $A$ and $b_1, b_2, ..., b_n$ the components in $b$. Besides, let 
$x_*$ denotes the optimal value such that it satisfies $Ax_* = b$.

In classical Kaczmarz algorithm \cite{kaczmarz1937angenaherte}, rows are iteratively picked in a sequential order. Denote $x_k$ to be the estimate of $x_*$ after $k$ iterations. At the $k$th updating step, we first pick one hyperplane $A_ix=b_i$ and then $x_{k}$ can be calculated by the former estimate $x_{k-1}$ and the picked hyerplane as follows. 
\begin{equation}
\label{eq:1}
x_{k} = x_{k-1} + \frac{b_i - \langle A_i, x_{k-1} \rangle}{\norm[A_i]_2^2}A_i
\end{equation}

Our work mainly relies on RKA-JL algorithm \cite{eldar2011acceleration} and RKA-Block algorithm \cite{needell2014paved}.
We will review these algorithms below.

\subsection{RKA-JL}

During each updating step, Kaczmarz algorithm chooses a hyperplane to project on. Thus, it is a combinatoric problem if we want to select out the optimal sequence of projecting hyperplanes and achieve the fastest convergence to the exact solution $x_*$. Unfortunately, such problem is too complicated to be solvable in reasonable amount of time. 
But, some greedy algorithm whose running time is affordable may be used to approximate the optimal sequence of hyperplanes and achieve a quite excellent convergence rate. 

Before introducing a greedy algorithm proposed by \cite{eldar2011acceleration}, we have to note a key property shared by all Kaczmarz related algorithms. Recall that Kaczmarz algorithm iteratively update the current estimate to a new one by projecting on to one hyperplane. 
It is easy to observe that the Euclidean distance between the current estimate and the solution is monotonically decreased, which means that
$$\norm[x_{k+1} - x_*]_2 \leq \norm[x_k - x_*]_2$$
This is obvious because
\begin{equation*}
\begin{aligned}
\norm[x_{k+1} - x_*]_2^2 &= \norm[x_k - x_*]_2^2 -\norm[x_{k+1} - x_k]_2^2 \\
						 &\leq \norm[x_k - x_*]_2^2
\end{aligned}
\end{equation*}
Thus, if we can find a way to maximize $\norm[x_{k+1} - x_k]_2$ at each iteration, we may achieve better convergence rate at last.

The exact greedy idea has been highlighted in \cite{eldar2011acceleration} when they proposed RKA-JL algorithm. The idea is quite simple but reasonable.
At the $k$-th updating step, we can choose the hyperplane $A_ix=b_i$ to project on where $A_i = \arg\max_{A_i} \norm[x_{k} - x_{k-1}]_2$ and update the estimate as Eq(\ref{eq:1}). 
But, after thinking of the practical issue, we may find that it is unaffordable to sweep through all rows of $A$ to pick the best one at each iteration. \
Thus, \cite{eldar2011acceleration} proposed a procedure to approximate the best hyperplane before performing each updating step, which is the key component of RKA-JL algorithm.

Instead of sweeping through the whole data and comparing the update $\norm[x_{k+1} - x_k]_2$, 
RKA-JL algorithm \cite{eldar2011acceleration} selects $p$ rows from $A$ with probability $\norm[A_i]_2^2 / \norm[A]^2_F$, 
utilizes Johnson-Lindenstrauss lemma to approximate the distance $\norm[x_{k+1} - x_k]_2 = \frac{|b_i - \langle A_i, x \rangle|}{\norm[A_i]_2}$ 
and then choose the maximized one as the row to determine the projecting hyperplane. 
Besides, a testing step is launched to ensure that the chosen hyperplane will not be worse than that of classical.

In the initialization step, the multiplication of $A$ and $\Phi$ costs $O(npd)$ time. 
And it takes $O(np)$ time to compute the sampling probability, which only needs to be computed once
and coule be used in the following selection steps during each updating step.
The selection step costs $O(pd)$ time,
while the testing and updating step clearly only	 cost $O(p)$ time.
Therefore each iteration costs $O(pd)$ time. 
The algorithm converge (in expectation) in $O(p)$ iterations, then the overall runtime is $O(npd + dp^2)$.

% Including the operations cost by computing sampling probablity and multiplying $\Phi$ in tnitialization step,
% it will cost $O(dnp + dp^2)$ flops. $d$ can be chosen on the order of $\log p$, then it will cost $O(np \log p + p^2 \log p)$ flops.

The whole process is given in Algorithm \ref{alg:RKA-JL}.
% In the next section, we will show that our Kaczmarz algorithm based on the clustering could lead an improvement in the greedy algorithm.
% \subsection{Randomized Kaczmarz Algorithm}
% \label{sec:RKA}
% 
% [[It has been observe in numerical simulations that the convergence rate of the Kaczmarz method can be significantly improved when \
% the algorithm sweeps through the rows of $A$ in a random manner, rather than sequentially in the given order.]]
% \cite{strohmer2009randomized} proposed Randomized Kaczmarz Algorithm(RKA) 
% selecting $A$'s rows $a_i$ with probability $\norm[a_i]^2 / \sum_{j=1}^n \norm[a_j]^2$. 
% 
% \begin{algorithm}[h]
% \caption{Randomized Kaczmarz Algorithm}
% \label{alg:RKA}
%  \begin{algorithmic}[1]
% 		\State $A \in \SetR^{n\times p}$, $b \in \SetR^{n}$. Solve consistent linear equation system $Ax = b$, where $x \in \SetR^{p}$.
% 		\For {$i = 1..numIter$}
% 			\State Select $i$th row with probability $p_i = \frac{\norm[a_i]^2}{\sum_{j=1}^n \norm[a_j]^2}$.
% 			\State Upstate each step by \[ x_{k+1} = x_k + \frac{b_i - \langle A_i, x_k \rangle}{\norm[a_i]^2}A_i \]
% 			\State Stop continuation when $\norm[x_{k+1} - x_k] \leq \epsilon$.
% 		\EndFor
%  \end{algorithmic}
% \end{algorithm}
% 
% \subsection{RKA-JL}

\begin{algorithm}[ht]
 \caption{RKA-JL}
 \label{alg:RKA-JL}
 \begin{algorithmic}[1]
  \State \textbf{Input:} $A \in \SetR^{n\times p}$, $b \in \SetR^{n}$. Solve consistent linear equation system $Ax = b$, where $x \in \SetR^{p}$.
  \State \textbf{Initialization Step:} Create a $d\times n$ Gaussian matrix $\Phi$ and set $\alpha_i = \Phi A_i$. Initialize $x_0$. Set $k = 0$.
  \State \textbf{Selection Step:}Select $p$ rows so that each row $A_i$ is chosen with probability $\norm[A_i]_2^2 / \norm[A]_F^2$.
					For each row selected, calculate
					\[ \gamma_i = \frac{|b_i - \langle \alpha_i, \Phi x_k \rangle|}{\norm[\alpha_i]_2} \] 
					and set $j = \arg\max_i \gamma_i$.
					
	\State \textbf{Test Step:} Select the first row $A_l$ out of the $n$, explicitly calculate 
					\begin{align*}
					 \hat{\gamma_j} = \frac{|b_j - \langle A_j, x_k \rangle|}{\norm[A_j]_2}  \quad \textrm{and} \quad 
					 \hat{\gamma_l} = \frac{|b_l - \langle A_l, x_k \rangle|}{\norm[A_l]_2}
					\end{align*}
					If $\hat{\gamma_l} > \hat{\gamma_j}$, set $j = l$.
	\State \textbf{Update Step:} Set 
					\[ x_{k+1} = x_k + \frac{b_j - \langle A_j, x_k \rangle}{\norm[A_j]_2^2} A_j \]
	\State 	Set $k = k+1$. Go to step 3 until convergence.
 \end{algorithmic}
\end{algorithm}

In fact, RKA-JL algorithm is only a weak approximation to the greedy idea mentioned above,
since it only takes into consideration a small fraction of data each time.
% Clustering method could give us an overview of data, approximating the maximized one in full data.
It is obvious that the greater value of $d$ and row selection number will give greater improvements,
but at the expense of greater computation cost at each iteration.
There is a trade-off between improvement and computational expense.
However, after utilizing clustering methods, more data can be used in selecting the best hyperplane 
while keeping the computational cost in an acceptable amount. 
We will illustrate our modified RKA-Cluster-JL algorithm in section 3.1.

\subsection{RKA-Block}

% In the real life, the linear equation system may be inconsistent, which means $Ax = b + e$, $e \neq \mathbf{0}$.
% Then with Kaczmarz algorithm we can not get the optimal value $x_*$.
% \cite{wang2015randomized} and \cite{needell2010randomized} have discussed this case in detail.
% And, they give a proven analysis to this inconsistent case, the algorithm will still have a linear convergence rate with an error factor:
% \[ \Expect[x_k - x] \leq \left( 1 - \kappa (A)^{-2} \right)^{k/2} \norm[x_0]_2 + \kappa(A) \gamma \]
% where $\kappa(A)$ is the same as above and $\gamma = \max_i \frac{|e_i|}{\norm[A_i]_2}$.
% 
% Many ideas and key observations to improve Kaczmarz based algorithms are continuously proposed by researchers recently. For example, \cite{eldar2011acceleration} highlighted a greedy idea on how to choose the hyperplane to project on and proposed a modified randomized Kaczmarz algorithm by incorporating the well-known Johnson-Lindenstrauss Lemma \cite{johnson1984extensions}. Each time, they randomly select $p$ rows of $A$ and use the hyperplane that maximizes the perpendicular distance to the current estimate. To lower the computational expense of calculating $p$ distances, they approximate these distances by utilizing Johnson-Lindenstrauss Lemma \cite{johnson1984extensions}.
% 
Another direction of researches related with Kaczmarz lies on the utilization of multiple rows of $A$ to update at each updating step.
In \cite{elfving1980block,needell2014paved} block versions of Kaczmarz algorithm are proposed. Instead of just using one hyperplane at each updating step, block Kaczmarz uses multiple hyperplanes. To be specific, when updating $x_k$, we may project the old estimate $x_{k-1}$ using $A_\tau$ which is a submatrix in $A$ and its corresponding $b_\tau$ via $x_k = x_{k-1}+(A_\tau)^\dagger (b_\tau-A_\tau x_{k-1})$.

\begin{algorithm}[ht]
 \caption{RKA-Block}
 \label{alg:RKA-Block}
 \begin{algorithmic}[1]
  \State \textbf{Input:} $A \in \SetR^{n\times p}$, $b \in \SetR^{n}$. Solve consistent linear equation system $Ax = b$, where $x \in \SetR^{p}$. 
  \State \textbf{Initialization Step:} Initialize $x_0$. Set $k=0$.
  \State \textbf{Selection Step:} Uniformly choose some $\tau$ rows to construct a matrix block $A_\tau$.
	\State \textbf{Update Step:} Set 
					\[ x_{k+1} = x_k + (A_\tau)^\dagger(b_\tau - A_\tau x_k) \]
	\State Set $k = k+1$. Go to step 3 until convergence.
 \end{algorithmic}
\end{algorithm}

Block Kaczmarz algorithm selects several hyperplanes and project $x_k$ onto the intersection of several hyperplanes.
This procedure acts exactly the same as projecting the current estimate in hyperplanes iteratively until convergence.
However, only one updating step is required to achieve that, while iteratively bouncing between these hyperplanes takes time. 
See Algorithm 2 for details.

While block Kaczmarz provably expected linear rate of convergence, 
it remains a problem on how to choose rows to construct blocks so that they are well-conditioned. 
We will show in section 3.2 that utilizing clustering information helps a lot.
Besides, theoretical guarantee is given to demonstrate the acceleration of our modified RKA-Cluster-Block algorithm.

%%%%%%%%%%%%%%%%%%%%%%%%%%%%%%%%%%%%%%%%%%%%%%
\section{Methodology and Theoretical Analysis}
%%%%%%%%%%%%%%%%%%%%%%%%%%%%%%%%%%%%%%%%%%%%%%
In this section we will introduce our accelerated algorithms and give some theoretical analysis.
Our observation is that high-dimensional Gaussian distributed data tends to stretch in nearly orthogonal clusters.
See theorem below for Theoretical proof.

\begin{restatable}{thm}{thmorthogonal}
 \label{thm:orthogonal-in-high-dimension}
 $u$ and $v$ are two vectors in $\SetR^{d}$. 
 Suppose each entry of $u$ and $v$ are sampled from Gaussian distribution $\mathcal{N}(0, \sigma^2)$.
 Then one has the probability inequality
 \[ P (\frac{|u^Tv|}{ \norm[u]_2 \norm[v]_2} \leq \epsilon )\ge \left(1-\frac{1}{\epsilon^2(1-\delta)^4d}\right) \left(1-6e^{-c\delta^2d}\right) \]
 where $\delta \in [0, 1]$ and $c$ is a fixed constant.
\end{restatable}

\begin{proof}(Sketch)
 Denote $E_u$ the event that $u$ holds inequality $\sigma(1-\delta) \sqrt{d} \leq \norm[u]_2 \leq \sigma(1+\delta)\sqrt{d}$.
 Via Bayes' theorem we have $P(|u^Tv| \leq \epsilon \norm[u]_2 \norm[v]_2) = P(|u^Tv| \leq \epsilon \norm[u]_2 \norm[v]_2 | \overline{E_u}\cap \overline{E_v})
 P(\overline{E_u}\cap \overline{E_v})$. 
 We apply Chebyshev's inequality to bound the probability $P(|u^Tv| \leq \epsilon \norm[u]_2 \norm[v]_2 | \overline{E_u}\cap \overline{E_v})$.
 Together with bound on the probability $P(\overline{E_u}\cap \overline{E_v})$, we finish the proof.
 See appendix \ref{app:orthogonal-in-high-dimension} for details.		
\end{proof}

Therefore, two randomly chosen items from a high-dimensional Gaussian distributed data have a high probability to be perpendicular.
Even though real data entities in high dimension are not ideally sampled from Gaussian distribution, 
they still tend to stretch along with different axes \cite{hopcroft2015foundation}. 
Thus, they can be grouped into different clusters where distances to each other are quite large.
Moreover, if the data actually can be embedded in a small dimensional space, 
the number of clusters should be close to the rank of the space, and thus quite small. 
Thus, if we can measure the property of these clusters within affordable amount of time, 
we can quickly master some knowledge on the entire data set.

% Gaussian distribution is a common distribution in real life, so it is a mild assumption.
% Suppose we have normalized, zero-centered and Gaussian distributed high-dimensional data,
% then if we perfrom a clustering algorithm by distance,
% the data items in the same cluster will closer than ones in different clusters.
% The more closely perpendicular the data items are, the further they are. 
% In other words, different clusters are nearly perpendicular to each other under the high-dimensional Gaussian distributed case.

Base on this observation, we give two accelerating algorithms via taking advantage of clustering.
The clustering algorithm needs not to be specific.
Any clustering algorithm with runtime no more than $O(npd+dp^2)$ is acceptable.
In practice, we use a k-means clustering method\cite{king2012online}.
Clustering $A$'s rows into $k$ cluster will cost $O(knp)$. 
In experiments, $k = 4$. It can be absorbed in the order of RKA-JL's computing operation numbers.
Therefore the runtime on clustering is acceptable.

\subsection{RKA-Cluster-JL}

To perform each projection, Kaczmarz select one hyperplane $A_ix=b_i$ to be projected on. We should notice that the hyperplane is uniquely determined by normal vector $A_i/\norm[A_i]_2$. To see this, recall Eq(\ref{eq:1}).
If we scale $A_i$ and $b_i$ by multiplying a constant $c$, it will not change the updating result.
Besides, since all hyperplanes go through the point $x_*$, we can uniquely calculate out $b_i=A_ix_*$. Thus, it is enough to only consider the normal vectors of rows of $A$ with unit length when choosing hyperplane to be projected on.

% Unfortunately this projecting hyperplane selecting problem is an NP problem. 
In this section, we will utilize clustering method to improve RKA-JL algorithm proposed by \cite{eldar2011acceleration}.
We conduct a clustering algorithm on $A$ to cluster the rows into $k$ clusters, each of which has a representive vector or cluster center $A_{c_j}$, $j\in\{1,2,...,k\}$.
% And we find the \textit{orthogonal set} $\{A_{c_i} | \langle A_{c_i}, A_{c_j} \rangle \leq \epsilon \}$ for each $A_{c_j}$.
At each iteration, we first choose the cluster representive vector $A_{c_j}$ which gives the maximized update $\norm[x_{k+1} - x_k]_2$.
Then in the $j$th cluster we randomly choose $p$ rows with probability proportional to $\norm[A_i]_2^2 / \norm[A]^2_F$.
This procedure could help us on better approximating the furthest hyperplane since good planes are more likely to be in the furthest cluster. 
% In statistics, we expect that by forgetting about bad clusters and only selecting plane in good cluster we will have more probability to find the optimal hyperplane.
The whole process is detailedly stated in Algorithm \ref{alg:RKA-Cluster-JL}.

\begin{algorithm}[ht]
 \caption{RKA-Cluster-JL}
 \label{alg:RKA-Cluster-JL}
 \begin{algorithmic}[1]
  \State \textbf{Input:} $A \in \SetR^{n\times p}$, $b \in \SetR^{n}$. 
  \State \textbf{Output:} Solve consistent linear equation system $Ax = b$, where $x \in \SetR^{p}$. 
  \State \textbf{Initialization Step:} Create a $d\times p$ Gaussian matrix $\Phi$ and set $\alpha_i = \Phi A_i$.
						Conduct a clustering algorithm in the rows of $A$, resulting in $c$ clusters with representive points $A_{\mathbf{c}_l}$, $l=\{1,2,...,c\}$.
						Initialize $x_0$. Set $k=0$.
  \State \textbf{Selection Step:} 
					Calculate 
					\[ \hat{x_k} = \Phi x_k \]
					For each representive point, calculate
					\[ r_l = \frac{|b_i - \langle A_{\mathbf{c}_l}, x_k \rangle |}{\norm[A_{\mathbf{c}_l}]_2} \]
					and set $t = \arg\max_l r_l$. \newline
					Select $p$ rows so that each row $A_i$ is chosen with probability $\norm[A_i]_2^2 / \norm[A]_F^2$ in the $t$th cluster.
					For each row selected, calculate
					\[ \gamma_i = \frac{|b_i - \langle \alpha_i, \Phi x_k \rangle|}{\norm[\alpha_i]_2} \] 
					and set $j = \arg\max_i \gamma_i$.
					
	\State \textbf{Test Step:} Select the first row $a_l$ out of the $n$, explicitly calculate 
					\begin{align*}
					 \hat{\gamma_j} = \frac{|b_j - \langle a_j, x_k \rangle|}{\norm[a_j]_2}  \quad \textrm{and} \quad 
					 \hat{\gamma_l} = \frac{|b_l - \langle a_l, x_k \rangle|}{\norm[a_l]_2}
					\end{align*}
					If $\hat{\gamma_l} > \hat{\gamma_j}$, set $j = l$.
	\State \textbf{Update Step:} Set 
					\[ x_{k+1} = x_k + \frac{b_j - \langle a_j, x_k \rangle}{\norm[a_j]_2^2} a_j \]
	\State 	Set $k \leftarrow k+1$. Go to step 4 until convergence.
 \end{algorithmic}
\end{algorithm}

The most difference between Algorithm \ref{alg:RKA-JL} and Algorithm \ref{alg:RKA-Cluster-JL} is that
Algorithm \ref{alg:RKA-Cluster-JL} chooses the best cluster representive point first, and then 
process as the same in Algorithm \ref{alg:RKA-JL}.

Ideally, if the furthest hyperplane to the current estimate lies in the furthest cluster, 
Algorithm \ref{alg:RKA-Cluster-JL} doesn't spend time to consider data points in other clusters at all,
while Algorithm \ref{alg:RKA-JL} still does that. 
If the total budget for sampled rows each time is fixed to $s$, 
Algorithm \ref{alg:RKA-Cluster-JL} spends $(s-c)$ of them searching in the right cluster, 
while Algorithm \ref{alg:RKA-JL} only spends $s/c$, which drops the probability
of it to find better hyperplane compared with Algorithm \ref{alg:RKA-Cluster-JL}.
% Roughly speaking, RKA-Cluster-JL needs to compare $c+p$ times 
% while RKA-JL needs to compare $cp$ times to get the same utility.

To theoretically analysize our algorithm, we propose the following proposition. 

\begin{restatable}{prop}{propJL}
\label{prop:cluster-accelerate-1}
 $A \in \SetR^{n\times p}$ is a row-normalized matrix, whose rows have unit length.
 Suppose the row vectors of $A$ are uniformly distributed in the high-dimensional space.
 Cluster these row vectors by directions into $k = O(\log(p))$ clusters, each of which has $t = \frac{n}{k}$ rows.
 Among $k$ clustering representive vectors, let $A_{c}$ be the one maximizing the update $\norm[x_{k+1} - x_k]_2^2$. 
 Suppose the rows in the $c$th cluster have bigger updates than rows in other clusters.
 In RKA-JL, it set $d = O(\log(p))$ for Gaussian matrix $\Phi \in \SetR^{d\times p}$.
 Then the utility of the RKA-JL algorithm comparing $kp$ rows to find a maximized one in $O((\log(p))^2 p)$ time
 is the same of the utility of  RKA-Cluster-JL algorithm comparing $k+p$ rows in $O(\log(p) p)$ time.
\end{restatable}
\begin{proof}
 See appendix \ref{app:cluster-accelerate-1} for details.
\end{proof}

Since the updates are only determined by the directions of the rows,
roughly speaking, the rows with similar directions in the same cluster will have similar updates.
Therefore the rows in cluster $c$ tends to have bigger updates than the rows in other cluster.
This is essentially the key idea behind this algorithm.

%%%%%%%%%%%%%%%%%%%%%%%%%%%%%%%%%%%%%%%%%%%%%%%%%%
\subsection{RKA-Cluster-Block}
%%%%%%%%%%%%%%%%%%%%%%%%%%%%%%%%%%%%%%%%%%%%%%%%%%
In this subsection, we will apply clustering method to block Kaczmarz.
We will show that by using the clustering information we can easily construct well-conditioned matrix blocks doing favor in the convergence analysis of block Kaczmarz.

\begin{restatable}{lem}{lemneedel}\cite{needell2014paved}.
\label{lem:Needell-convergence}
 Suppose $A$ is a matrix with full column rank that admits an $(m, \alpha, \beta)$ row  paving $T$.
 Consider the least-squares problem 
	\[ \min \norm[Ax - b]_2^2 \]
 Let $x_*$ be the unique minimizer, and define $e := Ax_* - b$.
 Then for randomized block Kaczmarz method, one has 
 \begin{equation*}
 \Expect[ \norm[ x_j - x_*]_2^2 ] \leq \left[ 1 - \frac{\sigma^2_{\min}(A)}{\beta m} \right] \norm[x_0 - x_*]_2^2  + \frac{\beta}{\alpha} \frac{\norm[e]_2^2}{\sigma^2_{\min}(A)} 
 \end{equation*}
 where an $(m,\alpha,\beta)$ row paving is a partition $T=\{\tau_1,...,\tau_m\}$ of row indices that satisfices
 \[ \alpha \leq \lambda_{\min}(A_{\tau_i} A_{\tau_i}^T) \quad \text{and} \quad \lambda_{\max}(A_{\tau_i} A_{\tau_i}^T) \leq \beta \]
 for each $\tau_i \in T$.
\end{restatable}
\begin{proof}
See appendix \ref{app:Needell-convergence} for details.
\end{proof}

From the lemma above, we notice that the convergence rate of block Kaczmarz algorithm highly depends on the spectral norm $\beta$ and condition number $\beta/\alpha$ of the block matrix.
The smaller value of $\beta$ and $\beta/\alpha$ will give us a faster convergence rate.
See Algorithm \ref{alg:RKA-Cluster-Block} for our entire proposed algorithm.

\begin{algorithm}
 \caption{RKA-Cluster-Block}
 \label{alg:RKA-Cluster-Block}
 \begin{algorithmic}[1]
  \State \textbf{Input:} $A \in \SetR^{n\times p}$, $b \in \SetR^{n}$. 
  \State \textbf{Output:} Solve consistent linear equation system $Ax = b$, where $x \in \SetR^{p}$. 
  \State \textbf{Initialization Step:}
		\begin{description}
		\item[Clustering:] Conduct a clustering algorithm in the rows of $A$, resulting in $c$ clusters with representive vectors $A_{\mathbf{c}_l}$, $l=\{1,2,...,c\}$
		\item[Partition:] Randomly extract one row of each cluster and compose to a row submatrix $A_{\tau i}$ from $A$, where $i \in \{1,2,...,T\}$. 
									 And denote the corresponding values as $b_{\tau i}$.
		\item[Setting:] $k = 0$, $x_k = 0$, $N$ the number of iteration.
		\end{description}
	\State \textbf{Selection Step:} Uniformly Select $A_{\tau i}$ from $\{ A_{\tau_1}, A_{\tau_2}, ..., A_{\tau_T} \}$.
	\State \textbf{Update Step:} Set 
					\[ x_{k+1} = x_k + (A_\tau)^\dagger(b_\tau - A_\tau x_k) \]
	\State Set $k = k+1$. Go to step 4 until convergence.
 \end{algorithmic}
\end{algorithm}

Next, we will theoretically show that our algorithm have a better convergence rate under a mild assumption that data are sampled according to high-dimensional Gaussian distribution.
The runtime analysis is similar to RKA-Cluster-JL.

To measure the orthogonality between matrix rows, we define the orthogonality value.
\begin{definition}[Orthogonality Value]
$A$ is an $k \times p$ matrix. 
Let $\hat{A}$ be the matrix after normalizing rows of $A$.
Then each row of $\hat{A}$ has unit length.
Define \textit{Orthogonality Value} $$ov(A) = \max_{i\neq j}| \langle \hat{A}_i, \hat{A}_j \rangle |$$  where $\hat{A}_i$ is the $i$th row of $\hat{A}$.
\end{definition}
Clearly, the inequality $0 \leq ov(A) \leq 1$ holds for any matrix $A$.
Take some examples to get a close look,
$ov(I) = 0$, $ov(\text{ones}\footnote{Matlab notation.}(5,5))
= ov(\text{ones}(5,5) / \sqrt{5}) = 1$.
Then we could give a upper bound on spectral norm below.

\begin{restatable}{thm}{thmupper}
\label{thm:max-singular-value-upper-bound}
 $A \in \SetR^{k \times p}$, $A_i \in\SetR^{p}$ is the $i$th row of $A$, $\norm[A_i]_2 = 1$ for all $i \in \{1,2,...,k\}$.
 Suppose the orthogonality value $ov(A) \leq \epsilon$,
 then one has 
 \[ \norm[AA^T]_2 \leq 1 + k \epsilon. \]   
\end{restatable}
\begin{proof}(Sketch)
 Using the fact that $\norm[X]_2^2 \leq \norm[X]_1 \norm[X]_\infty$,
 we bound $\norm[AA^T]_1$ and $\norm[AA^T]_\infty$ respectively.
 See appendix \ref{app:max-singular-value-upper-bound} for details.
\end{proof}

This theorem gives us an upper bound of spectral norm to the matrix with bound on orthogonality value. 
Lower orthogonality value corresponds to that $A$'s rows are almost perpendicular. 
Thus, we can conclude that selecting rows from each cluster to construct block matrices can give us small spectral norm.

\begin{restatable}{thm}{thmlower}
\label{thm:max-singular-value-lower-bound}
 Let $A$ be a $k \times p$ row-normalized matrix. 
 Suppose $| \langle A_i, A_j \rangle | \geq \delta$ for all $i, j \in \{1,2,..., k\}$,
 then one has 
 \[ \norm[AA^T]_2 \geq  1+(k-1)\delta. \]
\end{restatable}
\begin{proof}
 See appendix \ref{app:max-singular-value-lower-bound} for details.
\end{proof}

While, this theorem gives us an lower bound to the matrix with relatively low orthogonality value. 
This corresponds to the case that if we choose rows from only one or two clusters since these selected rows have almost same direction. 
In such case, we proved that the spectral norms of block matrices are quite large. 
Thus, it is quite reasonable to say that choosing rows from each cluster should converge faster than choosing rows from one or two clusters.

\begin{restatable}{thm}{thmcondition}
\label{thm:condition-number-upper-bound}
 $A \in \SetR^{k \times p}$, $A_i \in\SetR^{p}$ is the $i$th row of $A$, $\norm[A_i]_2 = 1$ for all $i \in \{1,2,...,k\}$.
 Suppose the orthogonality value $ov(A) \geq \frac{1}{\epsilon}$,
 then one can bound the condition number of $AA^T$,
 $$\mbox{cond}(AA^T) \leq \frac{1+k\epsilon}{1-\epsilon}.$$
\end{restatable}
\begin{proof}
 See appendix \ref{app:condition-number-upper-bound} for details.
\end{proof}

According to lemma \ref{lem:Needell-convergence}, 
spectral norm and condition number have a huge impact to the convergence of block Kaczmarz.
Based on the same assumption about low orthogonality value,
the two theorems \ref{thm:max-singular-value-upper-bound} and \ref{thm:condition-number-upper-bound}  give a theoretical analysis to the upper bounds of spectral norm and condition number.
Also, the experiments show that block Kaczmarz with clustering is more robust in the noisy case.

It is necessary to notice that \cite{needell2014paved} proposed two algorithms to select block matrices. 
One approach is an iterative algorithm repeatedly extracting a well-conditioned row-submatrix from $A$ until the paving is complete.
This approach is based on the column subset selection method proposed by \cite{tropp2009column}.
The another approach is a random algorithm partitioning the rows of $A$ in a random manner.
The iterative algorithm will give a set of well-conditioned blocks, but at a much expense of computation.
The random algorithm is easy to implement and bears an upper bound $\beta \leq 6\log(1+n)$ with high probability $1 - n^{-1}$.
Our construction for the clustering matrix blocks is more similar to the random algorithm in that we also construct the matrix block in a random manner, 
after clustering.

But to get the lower bound of $\alpha$, the random algorithm needs a fast incoherence transform, 
which changes the original problem $\min \norm[Ax - b]^2_2$ into 
$\norm[\tilde{A}x - \tilde{b}]_2^2$, where $\tilde{A} = SA$, $\tilde{b} = S b$ and $S$ is the fast incoherence transforming matrix.
Therefore it brings more noise into the original problem.
Without changing the original form of the least square problem, our clustering method will construct 
well-conditioned matrix block with proven upper bounds.

\section{Experiment}

In this section we empirically evaluate our methods in comparison with RKA-JL and RKA-Block algorithms.
We mainly follow \cite{needell2014paved} to conduct our experiments.
The experiments are run in a PC with WIN7 system, i5-3470 core $3.2$GHz CPU and $8$G RAM.

\begin{figure}[ht]
\centering
\subfigure[Add noise $\mathcal{N}(0, 0.1)$]{
  \includegraphics[width=0.22\textwidth]{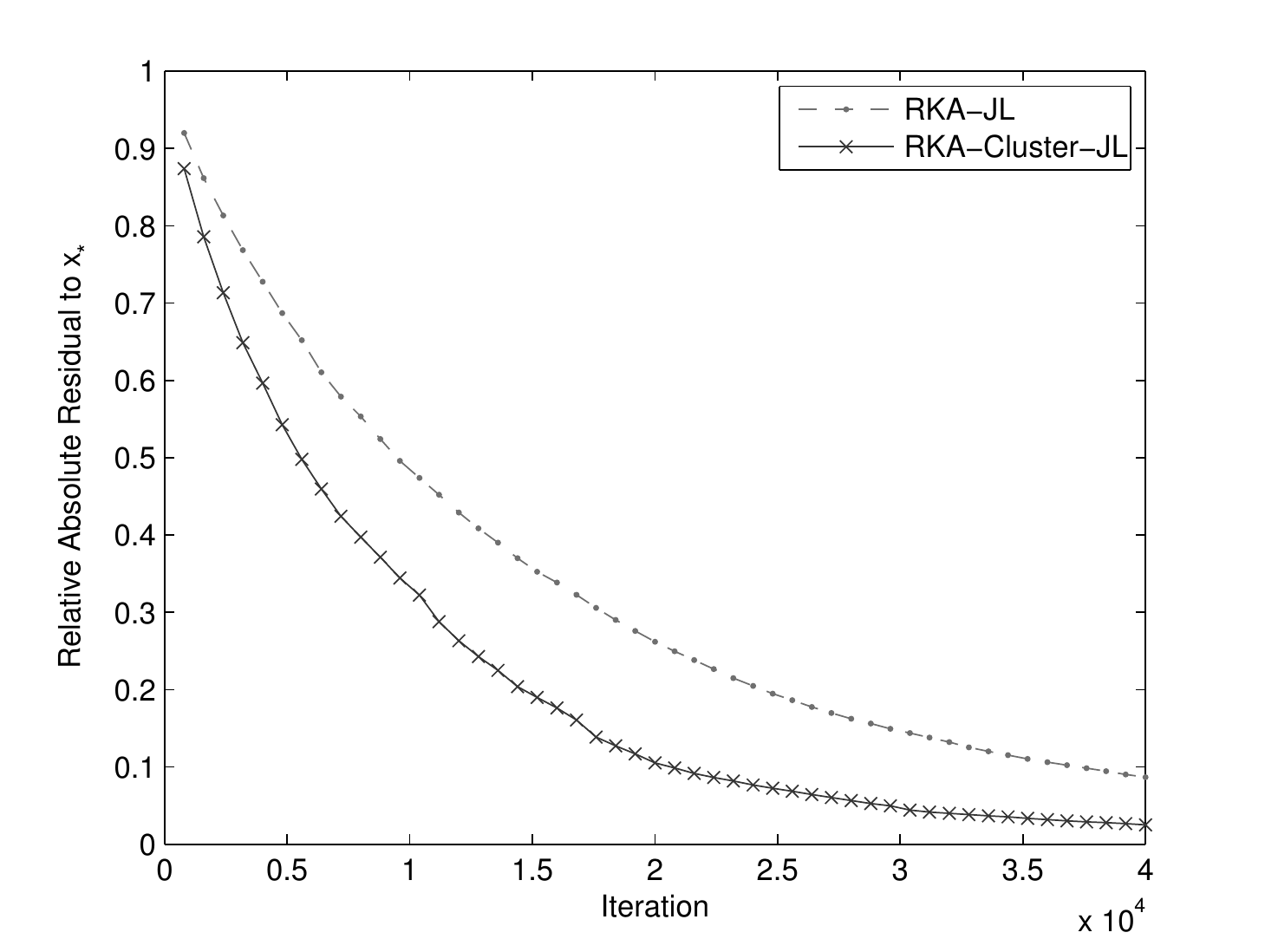}
}
\subfigure[Add noise $\mathcal{N}(0, 0.2)$]{
  \includegraphics[width=0.22\textwidth]{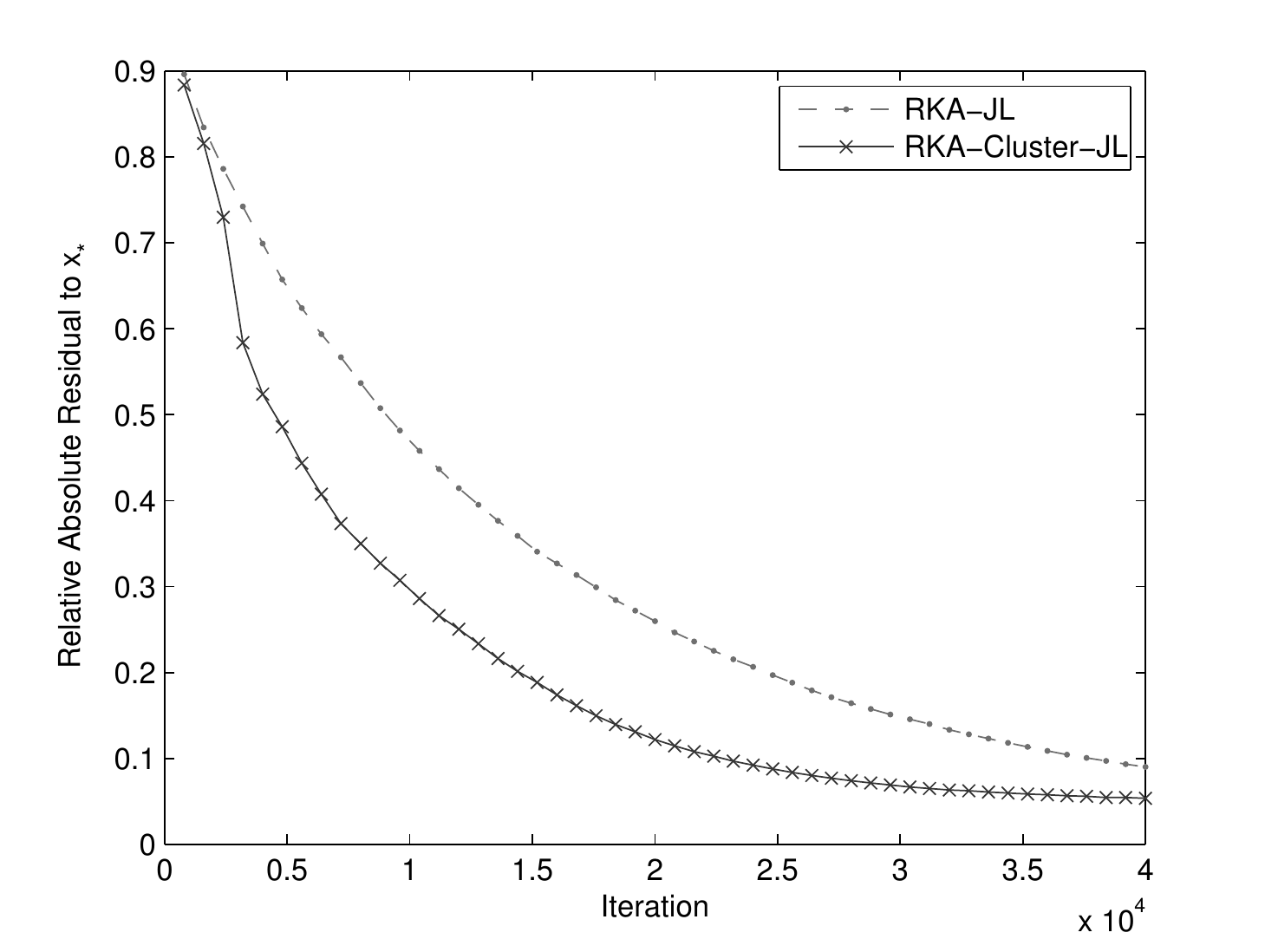}
}
\caption{Convergence comparison between RKA-JL and RKA-Cluster-JL}
\label{fig:jl}
\end{figure}

First, we compare the proposed RKA-Cluster-JL with the original RKA-JL algorithm. 
We generate data that comprises of several clusters. Here, $n=10000$ and $p=1000$. 
Besides, since the real data is usually corrupted by white noise, 
we add Gaussian noise with mean $0$ and standard deviation $0.1$ or $0.2$. 
Figure \ref{fig:jl} below shows that RKA-Cluster-JL outperforms RKA-JL. 
To cluster the data, we use the K-means variant algorithm \cite{king2012online}.

\begin{figure}[ht]
\centering
\subfigure[Add noise $\mathcal{N}(0, 0.1)$]{
  \includegraphics[width=0.22\textwidth]{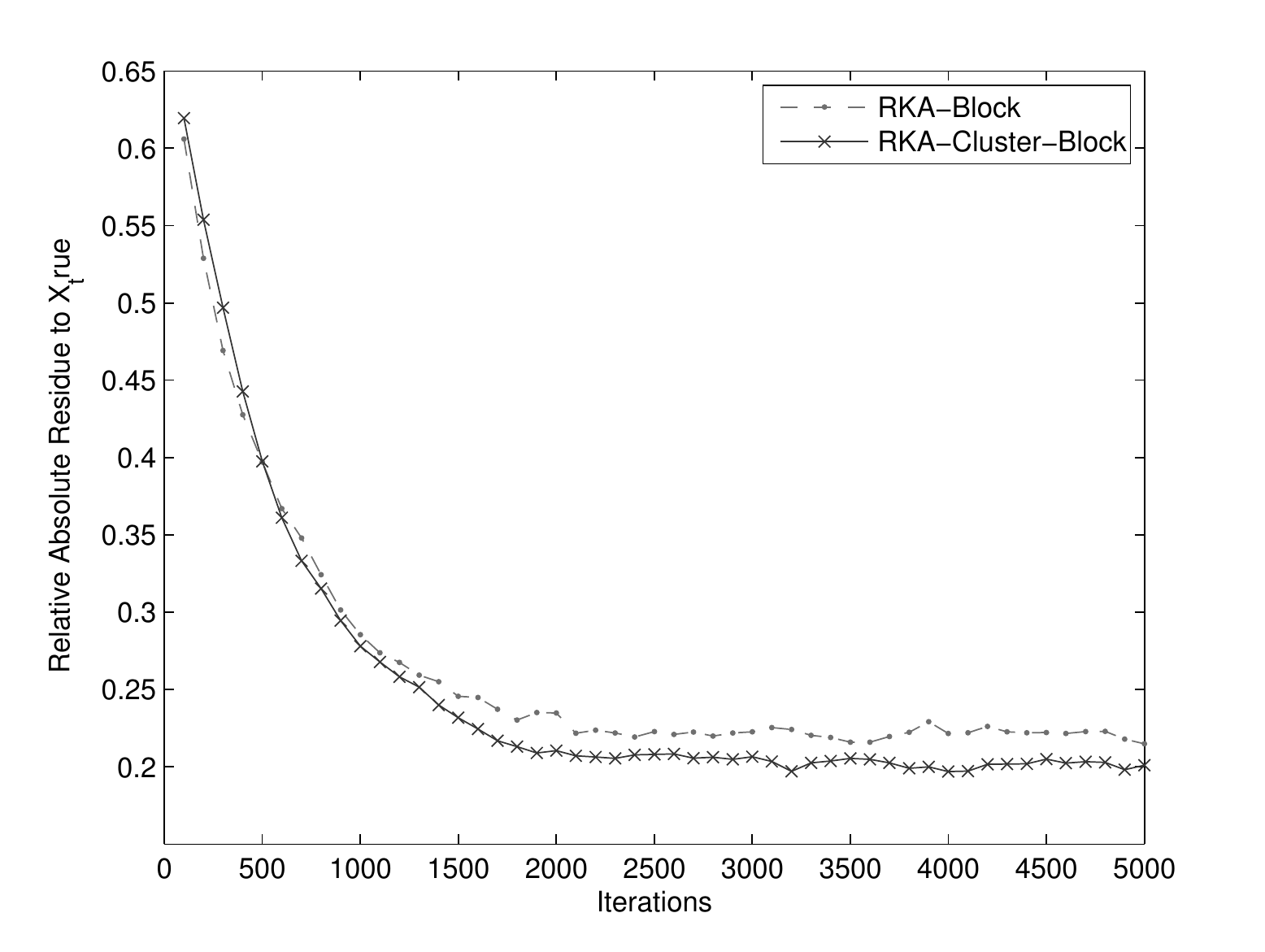}
}
\subfigure[Add noise $\mathcal{N}(0, 0.2)$]{
  \includegraphics[width=0.22\textwidth]{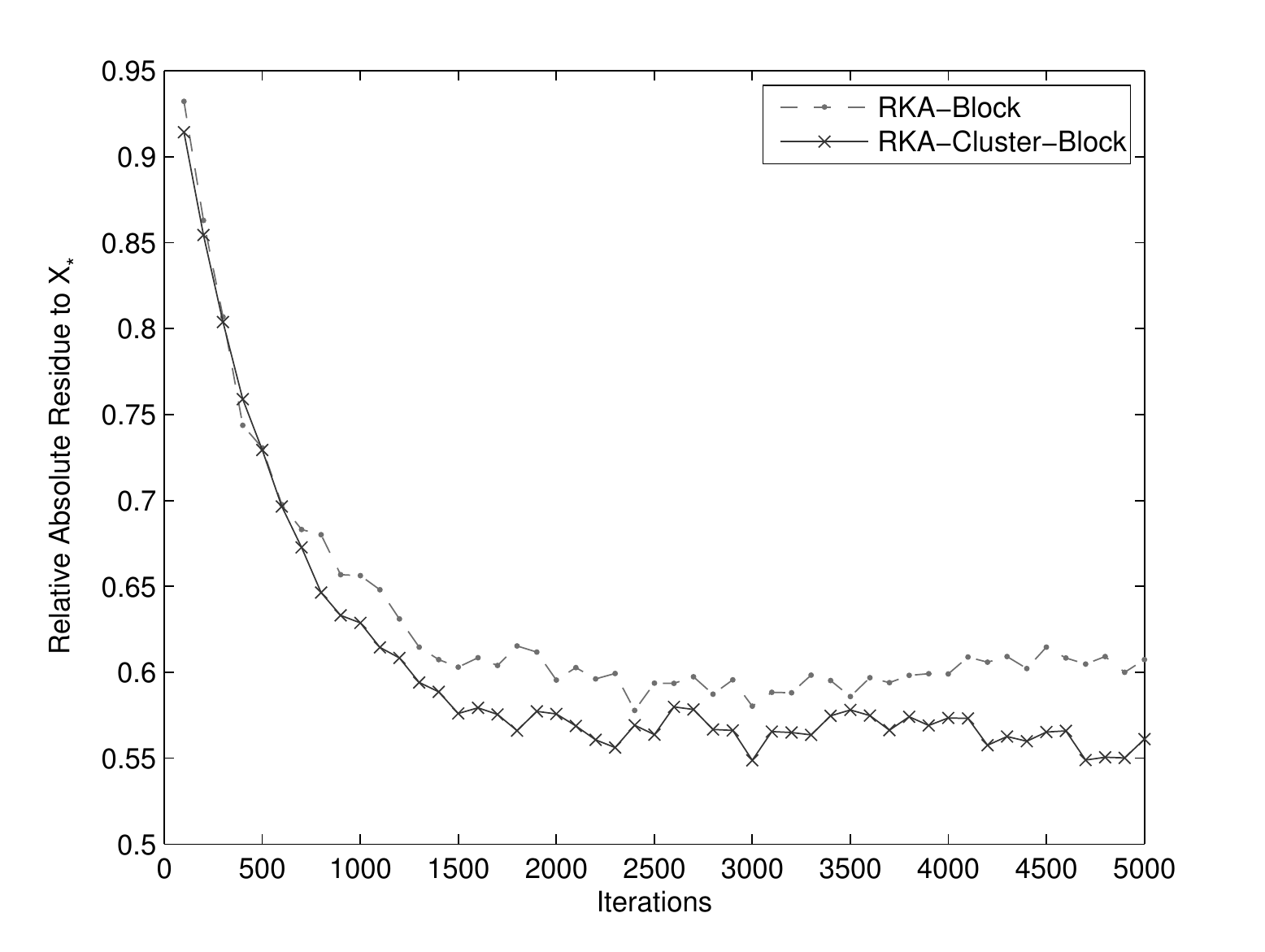}
}
\caption{Convergence comparison between RKA-Block and RKA-Cluster-Block}
\label{fig:block}
\end{figure}

Next, we compare the results produced by RKA-Block and RKA-Cluster-Block algorithms. 
We generate data that lies in four distinctive clusters. Here, $n=10000$, $p=1000$ and the block size is four. 
Then, as usual, white Gaussian noise with mean $0$ and standard deviation $0.1$ or $0.2$ is added to the data to simulate the real world. 
From Figure \ref{fig:block} below, we can see that our algorithm performs better than RKA-Block algorithm.

\begin{figure}[ht]
\centering
\subfigure[$\mathcal{N}(0, 0.1)$]{
  \includegraphics[width=0.16\textwidth]{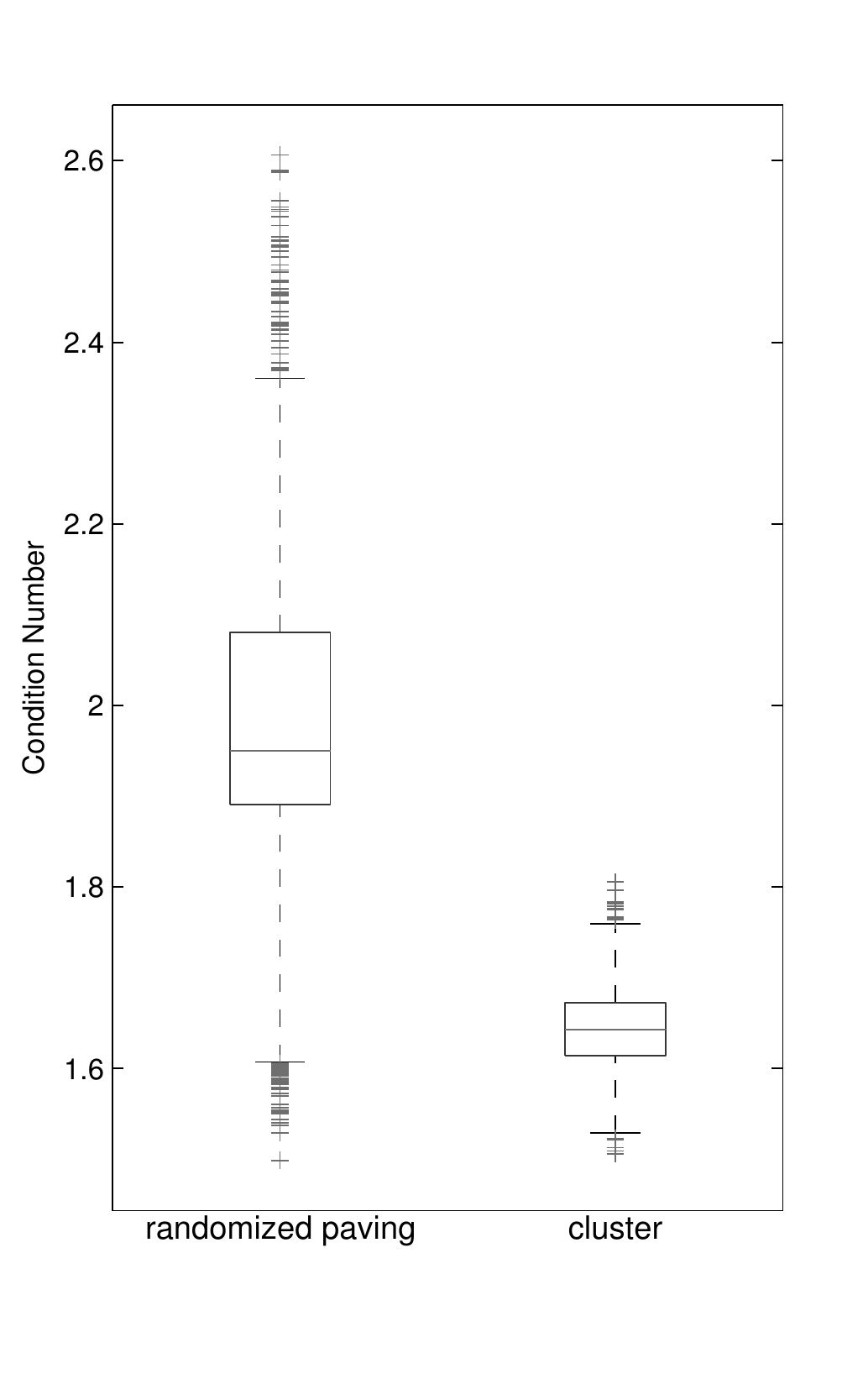}
}
\subfigure[$\mathcal{N}(0, 0.1)$]{
  \includegraphics[width=0.16\textwidth]{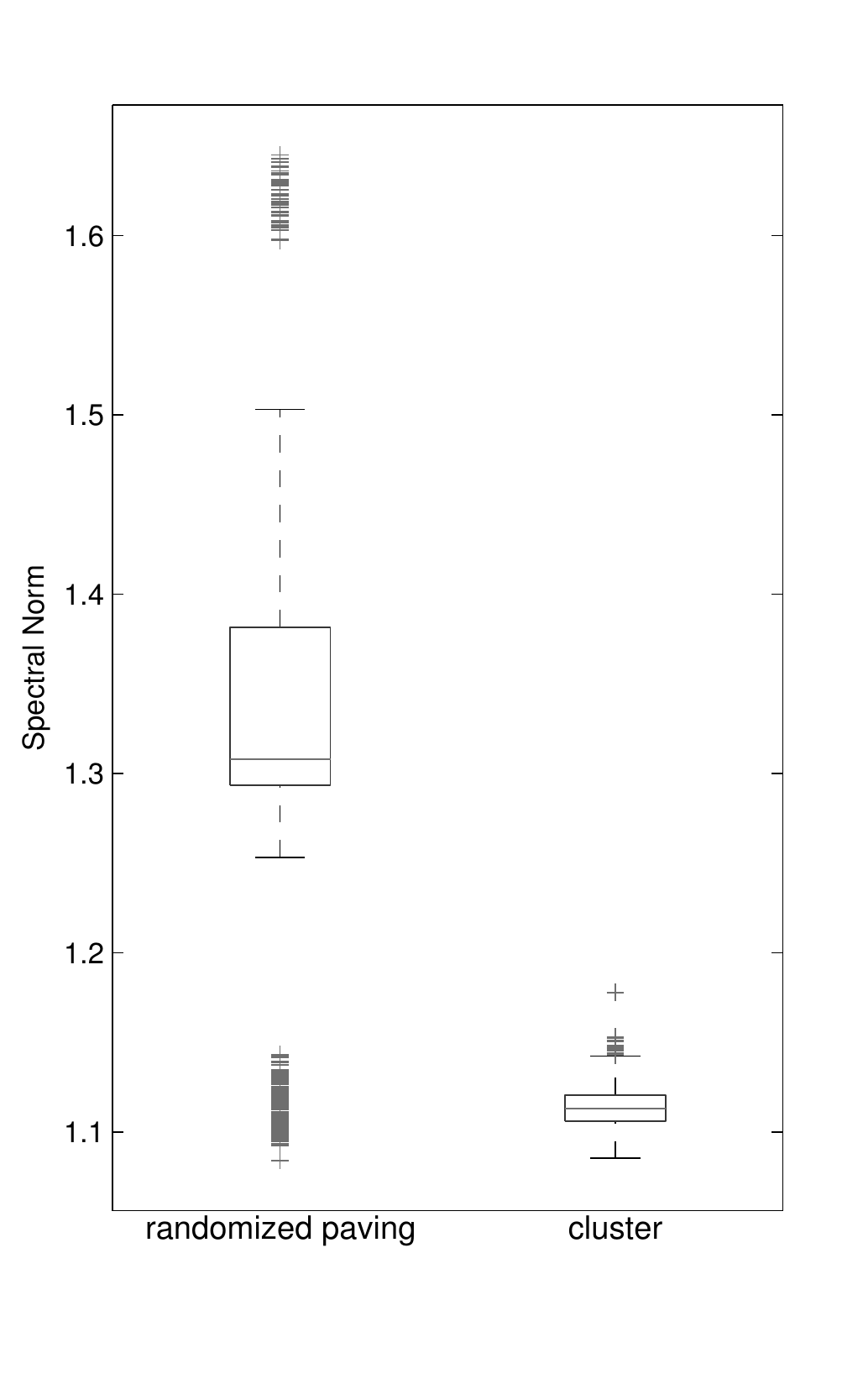}
}
\subfigure[$\mathcal{N}(0, 0.2)$]{
  \includegraphics[width=0.16\textwidth]{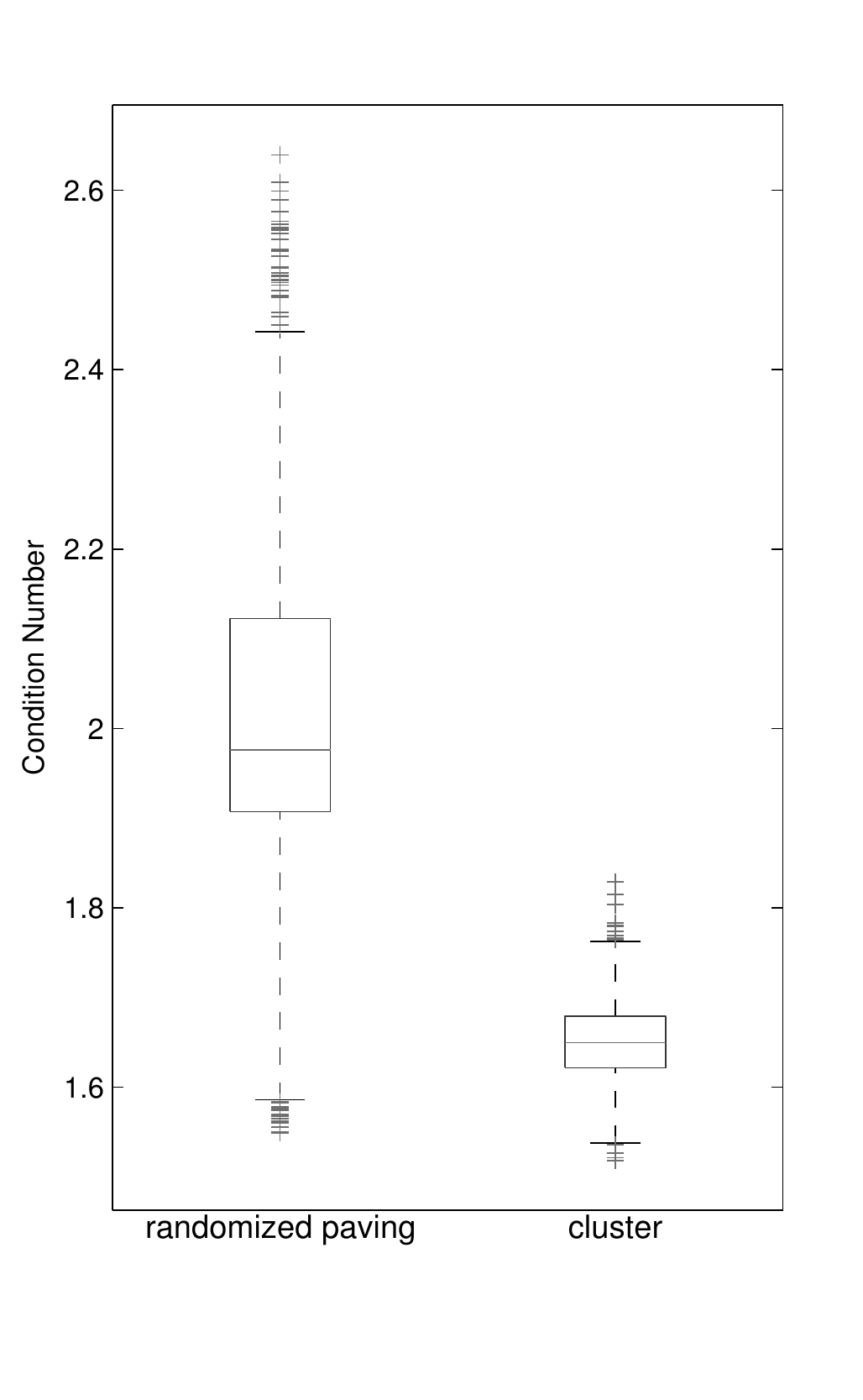}
}
\subfigure[$\mathcal{N}(0, 0.2)$]{
  \includegraphics[width=0.16\textwidth]{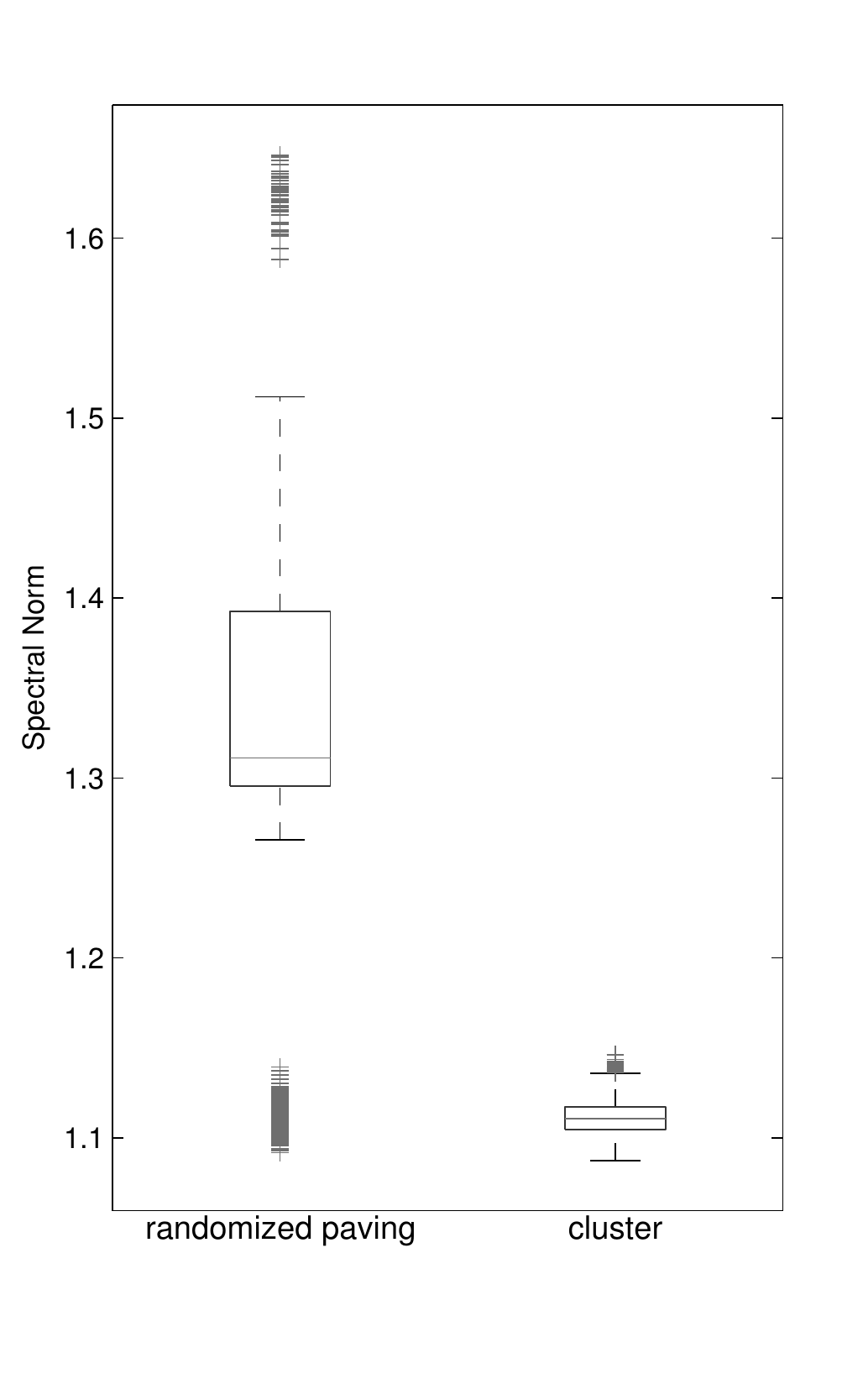}
}
\caption{Comparison between RKA-Block and RKA-Cluster-Block in Condition number and spectral norm}
\label{fig:condition-spectral}
\end{figure}

It is quite reasonable that our algorithm is better. 
Since our theoretical analysis tells us that smaller condition number or spectral norms of the block matrix, 
better performance the algorithm will have. 
We collect the condition numbers and spectral norm when these iterative algorithm running, 
and draw the box plots on Figure \ref{fig:condition-spectral}.
It shows that our algorithm gives both smaller condition number and spectral norm.

There is an implementation detail to mention. When the $k$ clusters are not of the same size, the number of matrix block will be constrained by the minimum size of cluster. In other words, we do not take all rows of $A$ into consideration.  
It is obvious that not considering all data will make our algorithm lose some information.
Thus, to avoid this potential problem when implementing RKA-Cluster-Block algorithm, we use the following procedure to construct matrix blocks. After we construct $p$ blocks, if there are many rows left, we continue using the left $k-1$ clusters to construct matrix block of size $k-1$. This will guarantee that our algorithm take all data into consideration.

\section{Conclusion}

In this paper,  
we propose an acceleration approach to improving the RKA-JL algorithm.
Our approach is based on a simple yet effective idea for clustering data points. 
Moreover, we have extended the clustering idea to construct well-conditioned matrix blocks,
which shows improvement over the block Kaczmarz. When data points follow a Gaussian distribution,
we have conducted theoretical and empirical analysis of our approaches.

\bibliography{reference}
\bibliographystyle{aaai}

\newpage
\clearpage
\appendix

\section{Proof for Theorem \ref{thm:orthogonal-in-high-dimension}}
\label{app:orthogonal-in-high-dimension}

\begin{lem}[Gaussian Annulus Theorem]
\label{lem:annulus}
For a $d$-dimensional spherical Gaussian with unit variance in each direction, 
for any $\beta\le \sqrt{d}$, 
all but at most $3e^{-c\beta^2}$ of the probability mass lies within the annulus $\sqrt{d}-\beta\le |r| \le \sqrt{d}+\beta$, 
where $c$ is a fixed positive constant and $r$ is the Euclidean distance between the point and the origin.
\end{lem}
This lemma is proposed by \cite{hopcroft2015foundation}

\thmorthogonal*

\begin{proof}

Expectation:
\begin{equation}
\mathbf{\Expect}(u^Tv)=\sum_{i=1}^d \Expect(u_i)\Expect(v_i)=0
\end{equation}

Variance:
\begin{equation}
\begin{aligned}
\Variance(u^Tv)&=\sum_{i=1}^d \Variance(u_iv_i) \\
&=d\cdot \Variance(u_iv_i)
\end{aligned}
\end{equation}

Since
\begin{equation}
\begin{aligned}
\Variance(u_iv_i) &= \mathbf{\Expect}(u_i^2v_i^2)-\mathbf{\Expect}^2(u_iv_i) \\
& = \mathbf{\Expect}(u_i^2)\mathbf{\Expect}(y_i^2) \\
&= \sigma^2\cdot \sigma^2 = \sigma^4
\end{aligned}
\end{equation}
we know that 
\begin{equation}
\Variance(u^Ty) = d\sigma^4
\end{equation}

According to \textbf{Gaussian Annulus Theorem} (Lemma \ref{lem:annulus}), which states that $\forall \beta\le \sqrt{d}$, $\sqrt{d}-\beta\le \norm[u]_2\le \sqrt{d}+\beta$ with at least probability $1-3e^{-c\beta^2}$.

Let $\beta = \frac{\sqrt{d}}{k} = \delta\sqrt{d}$.We know that 
\begin{equation}
\begin{aligned}
P(E_u) & \triangleq P\left(\sigma(1-\delta)\sqrt{d})\le \norm[u]_2\le \sigma(1+\delta)\sqrt{d}\right) \\
              & \ge 1-3e^{-c\delta^2d}
\end{aligned}
\end{equation}
where $E_u$ denotes the event that $u$ holds the above inequality.

For two points $u,v\in \mathbb{R}^d$, using \textbf{Union Bound}, we know that
\begin{equation}
\begin{aligned}
P(E_u \cup E_v) &\le Pr(E_u)+Pr(E_v) \\
&= 2\cdot 3e^{-c\delta^2d}
\end{aligned}
\end{equation}

Thus, we know that
\begin{equation}P\left(\overline{E_u}\cap \overline{E_v}\right) \ge 1-6e^{-c\delta^2d}
\end{equation}

Next, let's compute
\begin{equation}
P\left(\frac{|u^Tv|}{\norm[u]_2\norm[v]_2}\le \epsilon\right)
\end{equation}

\begin{equation}
\begin{aligned}
& P\left(\frac{|u^Tv|}{\norm[u]_2\norm[v]_2} \le \epsilon\right) \\
\le & P\left(\frac{|u^Tv|}{\norm[u]_2\norm[v]_2}\le \epsilon\cap (\overline{E_u}\cap \overline{E_v})\right) \\
=&P\left(\frac{|u^Tv|}{\norm[u]_2\norm[v]_2}\le \epsilon \Big|\overline{E_u}\cap \overline{E_v}\right)\cdot P(\overline{E_u}\cap \overline{E_v})
\end{aligned}
\end{equation}

To compute the first term, we use \textbf{Chebyshev's Ineuqality} as follows.
\begin{equation}
\begin{aligned}
&\ \ \ \ P\left(\frac{|u^Tv|}{\norm[u]_2\norm[v]_2}\ge \epsilon \Big|\overline{E_u}\cap \overline{E_v}\right) \\
&=P\left(|u^Tv|\ge \epsilon\norm[u]_2\norm[v]_2 \Big|\overline{E_u}\cap \overline{E_v}\right) \\
&\le P\left(|u^Tv|\ge \epsilon\sigma(1-\delta)\sqrt{d}\cdot \sigma(1-\delta)\sqrt{d}\right) \\
&=P\left(|u^Tv|\ge \epsilon\sigma^2(1-\delta)^2d\right) \\
&=P\left(|u^Tv-\mathbf{\Expect}(u^Tv)|\ge \epsilon\sigma^2(1-\delta)^2d\right) \\
&\le \frac{\Variance(u^Tv)}{\epsilon^2\sigma^4(1-\delta)^4d^2} 
=\frac{\sigma^4}{\sigma^4\epsilon^2(1-\delta)^4d} =\frac{1}{\epsilon^2(1-\delta)^4d}
\end{aligned}
\end{equation}

Thus, ones has
\begin{equation}
P\left(\frac{|u^Tv|}{\norm[u]_2\norm[v]_2}\le \epsilon \Big|\overline{E_u}\cap \overline{E_v}\right)\ge 1-\frac{1}{\epsilon^2(1-\delta)^4d}
\end{equation}

Thus, 
\begin{equation}
P\left(\frac{|u^Tv|}{\norm[u]_2\norm[v]_2}\le \epsilon\right)\ge \left(1-\frac{1}{\epsilon^2(1-\delta)^4d}\right) \left(1-6e^{-c\delta^2d}\right)
\end{equation}

For example, if we set $\delta=0.5$, we have
\begin{equation}
P\left(\frac{|u^Tv|}{\norm[u]_2\norm[v]_2}\le \epsilon\right)\ge \left(1-\frac{16}{\epsilon^2d}\right) \left(1-6e^{-\frac{cd}{4}}\right)
\end{equation}

\end{proof}

\section{Proof for Proposition \ref{prop:cluster-accelerate-1}}
\label{app:cluster-accelerate-1}
\propJL*

\begin{proof}
 Since the rows in the $c$th cluster have bigger update than rows in other clusters, 
 we should reduce our searching range into the $c$th cluster.
 In RKA-JL, we randomly sample $kp$ rows with probability $\norm[A_i]_2^2 / \norm[A]_F^2 = \frac{1}{n}$.
 Then there will $p$ rows located in the $t$th cluster in expectation.
 In RKA-Cluster-JL, we first choose the optimal cluster, then select $p$ rows in that cluster.
 Via Johnson-Lindenstrauss lemma, the computational expense of RKA-JL will be $O(kp\log(p)) =  O(p(\log(p))^2)$ time.
 And RKA-Cluster-JL will spend $O(kp + p\log(p)) = O(p\log(p))$ time.
%  It is the same that we randomly select $p$ rows proportional to $\norm[a_i]^2$ in RKA-Cluster-JL.
\end{proof}

\section{Proof for Lemma \ref{lem:Needell-convergence}} 
We first give the convergence analysis which mainly follows the  analysis in \cite{needell2014paved}.
\label{app:Needell-convergence}

\lemneedel*

\begin{proof}
 According to block Kaczmarz updating rule, one has 
 \begin{equation*}
 \begin{aligned}
  	x_j &= x_{j-1} + A_{\tau}^\dagger (b_\tau - A_{\tau} x_{j-1}) \\
		&= x_{j-1} + A_{\tau}^\dagger (x_* - x_{j-1}) - A_{\tau}^\dagger e_\tau 
  \end{aligned}
 \end{equation*}
 Since $b_\tau = A_{\tau} x_* - e_\tau$, subtract $x_*$, one has 
 \[ 
				x_j - x_* = (I - A_{\tau}^\dagger A_{\tau})(x_{j-1} - x_*) - A_{\tau}^\dagger e_\tau 
 \]
 The range of $A_{\tau}^\dagger$ and the range of $I - A_{\tau}^\dagger A_{\tau}$ are orthogonal, so one has 
 \[ 
				\norm[x_j - x_*]_2^2 = \norm[(I - A_{\tau}^\dagger A_{\tau})(x_{j-1} - x_*)]_2^2 + \norm[A_{\tau}^\dagger e_\tau]_2^2 
	\]
 The second term on the right-hand side satisfies 
 \[ 
				\norm[A_{\tau}^\dagger e_\tau]_2^2 \leq \sigma_{\max}^2 (A_{\tau}^\dagger) \norm[e_\tau]_2^2 \leq \frac{1}{\sigma_{\min}^2(A_{\tau})}\norm[e_\tau]_2^2
	\]
	
 Since $A_{\tau}^\dagger A_{\tau}$ is an projector, it is easy to verify $A_{\tau}^\dagger A_{\tau} = A_{\tau}^\dagger A_{\tau} A_{\tau}^\dagger A_{\tau}$, 
 $(A_{\tau}^\dagger A_{\tau})^T = A_{\tau}^\dagger A_{\tau}$.
 Denote $u = x_{j-1} - x_*$, then one has
 \begin{align*}
  \norm[(I-A_{\tau}^\dagger A_{\tau})u]_2^2 	&= \norm[u]_2^2 + \norm[A_{\tau}^\dagger A_{\tau} u]_2^2 - 2 u^T A_{\tau}^\dagger A_{\tau} \\
																				&= \norm[u]_2^2 - \norm[A_{\tau}^\dagger A_{\tau} u]_2^2
 \end{align*}
 For the second term in the right-hand side, one has 
 \begin{align*}
  \norm[A_{\tau}^\dagger A_{\tau} u]_2^2 	&\geq \sigma_{\min}^2 (A_{\tau}^\dagger) \norm[A_{\tau} u]_2^2 	\\
																		&= \frac{\norm[A_{\tau} u]_2^2}{\sigma_{\max}^2 (A_{\tau})} 		
 \end{align*}
 Take expectation on both sides, one has 
 \begin{align*}
		\Expect[ \norm[A_{\tau}^\dagger A_{\tau} u]_2^2 ] 	&\geq \frac{1}{\sigma^2_{\max}} \Expect[ \norm[A_{\tau} u]_2^2 ] \\
																									&= \frac{1}{\sigma^2_{\max}} \frac{1}{m} \sum_\tau \norm[A_{\tau} u]_2^2 \\
																									&= \frac{1}{\sigma^2_{\max}} \frac{1}{m} \norm[A u]_2^2 \\
																									&\geq \frac{\sigma^2_{\min}(A)}{m \sigma^2_{\max}} \norm[u]_2^2 \\
 \end{align*}
 
 So,
 \[ \Expect[ \norm[x_j - x_*]_2^2 ] \leq (1 - \frac{\sigma_{\min}^2(A)}{m \beta} ) \norm[x_{j-1} - x_*]_2^2 + \frac{1}{m\alpha} \norm[e]_2^2 \]
 Then it is easy to extend this inequality to 
 \begin{equation*}
 \Expect[ \norm[x_j - x_*]_2^2 ] \leq  \left(1 - \frac{\sigma_{\min}^2(A)}{m \beta} \right) \norm[x_{0} - x_*]_2^2   + \frac{\beta}{\alpha\sigma^2_{\min}(A)} \norm[e]_2^2 
 \end{equation*}
\end{proof}

%%%%%%%%%%%%%%%%%%%%%%%%%%%%%%%%%%%%%%%%%%%%%%%%%%%%%%%%%%%%%%%%%%%%%%%
\section{Proof for Theorem \ref{thm:max-singular-value-upper-bound}}
%%%%%%%%%%%%%%%%%%%%%%%%%%%%%%%%%%%%%%%%%%%%%%%%%%%%%%%%%%%%%%%%%%%%%%%
\label{app:max-singular-value-upper-bound}

First we introduce a lemma and corollary proposed in \cite{byrne2009bounds}.
\begin{lem}
 Let $A$ be a $k \times p$ matrix.
 Then for any $a \in [0,2]$, no eigenvalue of the matrix $A^\dagger A$ exceeds the maximum of 
 \[ \sum_{j=1}^p c_{aj} |A_{ij}|^{2-a} \]
 over all $i$, where $c_{aj} = \sum_{i=1}^k |A_{ij}|^a$.
\end{lem}

\begin{cor}
\label{cor:l2-norm-upper-bound}
 Selecting $a = 1$, one has 
 \[ \norm[A]_2^2 \leq \norm[A]_1 \norm[A]_\infty \]
\end{cor}

\thmupper*
\begin{proof}
 Let $B = AA^T$, then the diagonal item $B_{ii} = \langle A_i, A_i \rangle = 1$ and $B_{i,j} = \langle A_i, A_j \rangle \leq \epsilon$.
 According to \ref{cor:l2-norm-upper-bound}, $\norm[B]_2^2 \leq \norm[B]_1 \norm[B]_\infty$. 
 So we condiser $\norm[B]_1$ and $\norm[B]_\infty$ in the following.
 
 $ \norm[B]_1 \leq  \frac{\norm[Bx]_1}{\norm[x]_1} $ for all $x$.
 Denote $B_i$ the $i$th row of $B$, $S_i = \sum_{j\neq i} x_i$.
 Then 
 \[ B_i x \leq x_i + \epsilon S_i \]
 and 
 \[\norm[Bx]_1 \leq \sum_i^k x_i + \epsilon \sum_i^k S_i \leq (1+ k\epsilon) \norm[x]_1 
 \]
 Therefore, \[\norm[B]_1 \leq 1 + k\epsilon.\]
 Thus, we have 
 \begin{equation}
 \begin{aligned}
 \norm[B x]_\infty &= \max_i \norm[B_i x]_1 \leq x_i + \epsilon S_i \\
 &\leq (\max_i x_i) + k\epsilon (\max_i x_i) \\
 &= (1+k\epsilon) \norm[x]_\infty. 
 \end{aligned}
 \end{equation}
 
 Therefore 
 \[ \norm[B]_\infty \leq 1 + k\epsilon. \]
 Thus, we get $\norm[B]_2 \leq \sqrt{\norm[B]_1\norm[B]_\infty} \leq 1+k\epsilon$.
\end{proof}

\section{Proof for Theorem \ref{thm:max-singular-value-lower-bound}}
\label{app:max-singular-value-lower-bound}

\thmlower*

\begin{proof}
 Let $B = AA^T$, then the diagonal item $B_{ii} = \langle A_i, A_i \rangle = 1$ and $B_{i,j} = \langle A_i, A_j \rangle \geq \delta$.
 $\norm[B]_2 \geq \frac{\norm[Bx]_2}{\norm[x]_2}$ for all $x$.
 Let $x = [1;1; ...; 1] \in \SetR^{k\times 1}$, then one has 
 \begin{align*} 
		\frac{\norm[Bx]_2}{\norm[x]_2} &= \frac{1}{\sqrt{k}} \norm[Bx]_2 \\
																	&\geq \frac{1}{\sqrt{k}} \sqrt{k (1+(k-1)\mathbf{\delta})^2}  \\
																	&= 1 + (k-1) \delta
 \end{align*} 
\end{proof}

\section{Proof for Theorem \ref{thm:condition-number-upper-bound}}
\label{app:condition-number-upper-bound}

\begin{lem}[Theorem 0 in \cite{hong1992lower}]
\label{lem:diagonal}
 For a matrix $A \in \SetR^{n\times n}$, one has 
 \begin{equation*}
 \sigma_{\min}(A) \geq \min_{1\leq k \leq n}  \left\{ \frac{1}{2}  \left( \sqrt{ 4 |a_{kk}|^2 + [r_k(A) - c_k(A)]_2^2 } -  [r_k(A) + c_k(A)] \right) \right\}
 \end{equation*}
 where $r_k(A) = \sum_{j\neq k} |a_{kj}|$ and $c_k(A) = \sum_{j\neq k} |a_{jk}|$, $a_{kj}$ is the item in $k$th row and $j$th column in $A$.
\end{lem}

\begin{thm}
\label{thm:min-singular-value-lower-bound}
 $A \in \SetR^{k \times p}$, $A_i \in\SetR^{p}$ is the $i$th row of $A$, $\norm[A_i]_2 = 1$ for all $i \in \{1,2,...,k\}$.
 Suppose $| \langle A_i, A_j \rangle | \leq \epsilon$ for all $i, j \in \{1,2,..., k\}$, then one has 
 \[ \sigma_{\min}(AA^T) \geq 1-\epsilon  \]
\end{thm}

\begin{proof}
\label{45}
 Let $B = AA^T$, then one has 
	\[ B_{ii} = \langle A_i, A_i \rangle = 1 \quad \text{and}  |B_{i,j}| = |\langle A_i, A_j \rangle| \leq \epsilon. \]
 Apply \ref{lem:diagonal}, we can get the inequality
 \begin{align*} 
	\sigma_{\min} (B) &\geq \min_{1\leq i \leq k} \left\{ \frac{1}{2} \left( \sqrt{ 4 }  -  2\epsilon \right) \right\}  \\
										&= 1 - \epsilon
 \end{align*}
\end{proof}

\thmcondition*
\begin{proof}
Using Theorem \ref{thm:max-singular-value-upper-bound} and Theorem \ref{thm:min-singular-value-lower-bound}, one can easily prove this theorem.
\end{proof}

\end{document}